\documentclass[oneside,english]{amsart}
\usepackage[T1]{fontenc}
\usepackage[latin9]{inputenc}
\usepackage{amsthm}
\usepackage{amstext}
\usepackage{amssymb}
\usepackage{esint}
\usepackage{graphicx}
\usepackage{enumerate}
\usepackage{amscd}

\makeatletter
\newtheorem{theorem}{Theorem}[section]

\newtheorem{corollary}[theorem]{Corollary}
\numberwithin{figure}{section}
\theoremstyle{definition}

\theoremstyle{remark}
\newtheorem{remark}[theorem]{Remark}

\numberwithin{equation}{section}
\makeatother

\usepackage{color}

\usepackage{babel}

	\DeclareMathOperator{\dist}{dist}
	\DeclareMathOperator{\loc}{loc}

    \DeclareMathOperator{\Imag}{Im}
    \DeclareMathOperator{\Real}{Re}

\begin{document}

\title[Estimates of Dirichlet eigenvalues]{Estimates of Dirichlet eigenvalues of divergent elliptic operators in non-Lipschitz domains }

\author{V.~Gol'dshtein, V.~Pchelintsev, A.~Ukhlov}
\begin{abstract}
We study spectral estimates of the divergence form uniform elliptic operators $-\textrm{div}[A(z) \nabla f(z)]$ with the Dirichlet boundary condition in bounded non-Lipschitz simply connected domains $\Omega \subset \mathbb C$. The suggested method is based on the quasiconformal composition operators
on Sobolev spaces with applications to the weighted Poincar\'e-Sobolev inequalities.
\end{abstract}
\maketitle
\footnotetext{\textbf{Key words and phrases:} Elliptic equations, Sobolev spaces, quasiconformal mappings.}
\footnotetext{\textbf{2010
Mathematics Subject Classification:} 35P15, 46E35, 30C65.}

\section{Introduction}

The present paper is devoted to spectral estimates of the divergence form uniform elliptic operators 
\begin{equation}\label{EllDivOper}
L_{A}=-\textrm{div} [A(z) \nabla f(z)], \quad z=(x,y)\in \Omega,
\end{equation}
with the Dirichlet boundary condition,
in bounded simply connected non-Lipschitz domains $\Omega \subset \mathbb C$ that satisfy to well-known in the geometric analysis  quasihyperbolic boundary conditions \cite{KOT01,KOT02}.  This class of domains represents an important subclass of Gromov hyperbolic domains \cite{BHK}. In such domains we consider the weak statement of the Dirichlet eigenvalue problem: a function $f$ solves the previous problem iff $f$ belongs to the divergent Sobolev space $W_0^{1,2}(\Omega, A)$ generated by the matrix $A$ and
\[
\iint\limits_\Omega \left\langle A(z )\nabla f(z), \nabla \overline{g(z)} \right\rangle dxdy
= \lambda \iint\limits_\Omega f(z)\overline{g(z)}~dxdy
\]
for all $g \in W_0^{1,2}(\Omega, A)$.

It is known \cite{M} that in bounded domains ${\Omega} \subset \mathbb C$ the  Dirichlet spectrum of the divergence form elliptic operators $L_A$ is discrete and can be written in the form of a non-decreasing sequence
\[
0<\lambda_1(A,{\Omega}) \leq \lambda_2(A,{\Omega}) \leq \ldots \leq \lambda_n(A,{\Omega}) \leq \ldots,
\]
where each eigenvalue is repeated as many time as its multiplicity.
Hence, according to the Min-Max Principle (see, for example, \cite{Henr}) the first Dirichlet eigenvalue of the elliptic operator in divergence form can be calculated by
\[
\lambda_1(A,\Omega)= \inf_{f \in W_0^{1,2}(\Omega, A) \setminus \{0\}} \frac{\iint\limits_{\Omega} \left\langle A(z)\nabla f(z),\nabla f(z)\right\rangle\,dxdy}{\iint\limits_{\Omega} |f(z)|^2~dxdy}\,.
\]

The classical upper estimate for the first Dirichlet eigenvalue of the Laplace operator ($A=I$) in the case of
simply connected planar domains with rectifiable boundaries was obtained by Payne and Weinberger \cite{PW}:
\[
\lambda_1(\Omega) \leq \frac{\pi j^2_{0,1}}{|\Omega|} \left[1+\left(\frac{1}{J_1^2(j_{0,1})}-1\right) \left(\frac{|\partial \Omega|^2}{4 \pi |\Omega|}-1\right)\right],
\]
where $|\Omega|$ is the Lebesgue measure of $\Omega$, $|\partial \Omega|$ is the one-dimensional Hausdorff measure of the boundary of $\Omega$ and $J_1$ denotes the Bessel function of the first kind of order one.

In the present article we prove that if a simply connected bounded domain $\Omega \subset \mathbb C$ satisfies to the quasihyperbolic boundary conditions, then
\begin{multline*}
\lambda_1(A,\Omega) \leq K \lambda_1(\mathbb D) + A^2_{\frac{4\beta}{\beta -1},2}(\mathbb D) K^2 \lambda_1^2(\mathbb D_\rho) \\
\times \left(\pi^{\frac{1}{2\beta}} +
\|J_{\varphi^{-1}}\,|\,L^{\beta}(\mathbb D)\|^{\frac{1}{2}} \right) \cdot
\|1-J_{\varphi^{-1}}^{\frac{1}{2}}\,|\,L^{2}(\mathbb D)\|,
\end{multline*}
where $\mathbb D_\rho$ is the largest disk inscribed in $\Omega$, the number $\beta$ is defined by the quasihyperbolic geometry of the domain and $A_{\frac{4\beta}{\beta -1},2}(\mathbb D)$ is the exact Poincar\'e-Sobolev constant \cite{GPU2019}. In this estimate
$\varphi:\Omega \to \mathbb D$ is a quasiconformal mapping associated with the matrix $A$ ($A$-quasiconformal mapping in our notation) via the Beltrami equation and $K$ its quasiconformality coefficient.

This estimate generalizes the upper estimate of the Dirichlet eigenvalues of the Laplace operator proved in \cite{GPU19A}.
Note, that domains that satisfy the quasihyperbolic boundary condition can have non-rectifiable boundaries \cite{GPU2020}. Let us consider, for example, von Koch snowflakes $\Omega\subset\mathbb C$, then it is known that for some such snowflakes the one-dimensional Hausdorff measure $|\partial \Omega| =\infty$ and $|\Omega|<\infty$. So, we obtain upper estimates of Dirichlet eigenvalues  in more wide class of domains than in  \cite{PW}.

Recall that a domain $\Omega$ satisfies the $\gamma$-quasihyperbolic boundary condition with some $\gamma>0$, if the growth condition on the quasihyperbolic metric
$$
k_{\Omega}(x_0,x)\leq \frac{1}{\gamma}\log\frac{\dist(x_0,\partial\Omega)}{\dist(x,\partial\Omega)}+C_0
$$
is satisfied for all $x\in\Omega$, where $x_0\in\Omega$ is a fixed base point and $C_0=C_0(x_0)<\infty$,
\cite{GM, H1, HSMV}.

In \cite{AK} it was proved that Jacobians $J_{\psi}$ of  quasiconformal mappings $\psi: \mathbb D \to \Omega$ belong to $L_{\beta}(\mathbb D)$ for some $\beta>1$ if and only if $\Omega$ satisfy to a $\gamma$-quasihyperbolic boundary conditions for some $\gamma$. Note that the degree of integrability $\beta$ depends only on
$\Omega$ and the quasiconformality coefficient $K(\psi)$.

Since we need the exact value of the integrability exponent $\beta$ for quasiconformal Jacobians, we consider the equivalent definition of domains satisfy the quasihyperbolic boundary condition in terms of integrability of Jacobians \cite{GPU19}. Let us remind that a simply connected domain $\Omega$ is called an $A$-quasiconformal $\beta$-regular domain about a simply connected domain $\widetilde{\Omega}$, $\beta >1$, if
$$
\iint\limits_{\widetilde{\Omega}} |J(w, \varphi^{-1})|^{\beta}~dudv < \infty,
$$
where $\varphi: \Omega\to\widetilde{\Omega}$ is a corresponding $A$-quasiconformal mapping.

The Ahlfors domains \cite{Ahl66} (quasidiscs \cite{VGR}) represent an important subclass of $A$-quasi\-con\-for\-mal $\beta$-regular domains. Moreover, in these domains spectral estimates can be specified in terms of the "quasiconformal geometry" of domains (Section~5).

The lower estimates of Dirichlet eigenvalues directly connected  to the isoperimetric Rayleigh-Faber-Krahn inequality which states that the disc minimizes the first Dirichlet eigenvalue of the Laplace operator (the $A$-divergent form elliptic operator with the matrix $A=I$) among all planar domains of the same area (see, for example, \cite{Henr}):
\begin{equation*}
\lambda_1(I,\Omega):=\lambda_1(\Omega)\geq \lambda_1(\Omega^{\ast})=\frac{{j_{0,1}^2}}{R^2_{\ast}},
\end{equation*}
where $j_{0,1} \approx 2.4048$ is the first positive zero of the Bessel function $J_0$ and $\Omega^{\ast}$ is a disc of the same area as $\Omega$ with $R_{\ast}$ as its radius.

The first lower bound for the first Dirichlet eigenvalue of the Laplace operator for a simply connected planar domain was obtained by Makai \cite{M65} and later was rediscovered (in a weaker form) by Hayman \cite{H78}:
\[
\lambda_1(\Omega)\geq \frac{\alpha}{\rho^2},
\]
where $\alpha$ is a constant and $\rho$ is the radius of the largest disc inscribed in $\Omega$.

Let $\Omega\subset\mathbb C$ be a pre-image of the unit disc $\mathbb D$ under an $A$-quasiconformal mapping $\varphi: \Omega\to\mathbb D$ with $|J(z,\varphi)|=1$, $z \in \Omega$. Some examples of such maps can be found in Section 3. In this case we have the Rayleihg-Faber-Krahn type estimates
\[
\lambda_1(\mathbb D) \leq \lambda_1(A,\Omega) \leq K \lambda_1(\mathbb D).
\]

The suggested method is based on connections between the quasiconformal mappings agreed with the matrix $A$ \cite{AIM,GNR18} and composition operators on Sobolev spaces \cite{GPU2020}. The composition operators theory \cite{GGu,GU,U93,VU02} allows to obtain lower and upper estimates of the first Dirichlet eigenvalues of elliptic operators in divergence form in planar domains with quasihyperbolic boundary conditions.

\section{Sobolev spaces and $A$-quasiconformal mappings}

Let $E \subset \mathbb C$ be a measurable set on the complex plane and $h:E \to \mathbb R$ be a positive a.e. locally integrable function i.e. a weight. The weighted Lebesgue space $L^p(E,h)$, $1\leq p<\infty$,
is the space of all locally integrable functions with the finite norm
$$
\|f\,|\,L^{p}(E,h)\|= \left(\iint\limits_E|f(z)|^ph(z)\,dxdy \right)^{\frac{1}{p}}< \infty.
$$

The two-weighted Sobolev space $W^{1,p}(\Omega,h,1)$, $1\leq p< \infty$, is defined
as the normed space of all locally integrable weakly differentiable functions
$f:\Omega\to\mathbb{R}$ endowed with the following norm:
\[
\|f\mid W^{1,p}(\Omega,h,1)\|=\|f\,|\,L^{p}(\Omega,h)\|+\|\nabla f\mid L^{p}(\Omega)\|.
\]

In the case $h=1$ this weighted Sobolev space coincides with the classical Sobolev space $W^{1,p}(\Omega)$.
The seminormed Sobolev space $L^{1,p}(\Omega)$, $1\leq p< \infty$,
is the space of all locally integrable weakly differentiable functions $f:\Omega\to\mathbb{R}$ endowed
with the following seminorm:
\[
\|f\mid L^{1,p}(\Omega)\|=\|\nabla f\mid L^p(\Omega)\|, \,\, 1\leq p<\infty.
\]

We also need a weighted seminormed Sobolev space $L^{1,2}(\Omega, A)$ (associated with the matrix $A$), defined
as the space of all locally integrable weakly differentiable functions $f:\Omega\to\mathbb{R}$
with the finite seminorm given by:
\[
\|f\mid L^{1,2}(\Omega,A)\|=\left(\iint\limits_\Omega \left\langle A(z)\nabla f(z),\nabla f(z)\right\rangle\,dxdy \right)^{\frac{1}{2}}.
\]

The corresponding  Sobolev space $W^{1,2}(\Omega, A)$ is defined
as the normed space of all locally integrable weakly differentiable functions
$f:\Omega\to\mathbb{R}$ endowed with the following norm:
\[
\|f\mid W^{1,2}(\Omega, A)\|=\|f\,|\,L^{2}(\Omega)\|+\|f\mid L^{1,2}(\Omega,A)\|.
\]
The Sobolev space $W^{1,2}_{0}(\Omega, A)$ is the closure in the $W^{1,2}(\Omega, A)$-norm of the
space $C^{\infty}_{0}(\Omega)$.

Recall that a homeomorphism $\varphi: \Omega \to \widetilde{\Omega}$, $\Omega,\, \widetilde{\Omega} \subset\mathbb C$, is called a $K$-quasiconformal mapping if $\varphi\in W^{1,2}_{\loc}({\Omega})$ and there exists a constant $1\leq K<\infty$ such that
$$
|D\varphi(z)|^2\leq K |J(z,\varphi)|\,\,\text{for almost all}\,\,z \in \Omega.
$$

Now we give a construction of $A$-quasiconformal mappings connected with the $A$-divergent form elliptic operators.
We suppose that matrix functions $A(z)=\left\{a_{kl}(z)\right\}$  with measurable entries $a_{kl}(z)$ belongs to a class  $M^{2 \times 2}(\Omega)$ of all $2 \times 2$ symmetric matrix functions that satisfy to an additional condition $\textrm{det} A=1$ a.e. and to the uniform ellipticity condition:
\begin{equation}\label{UEC}
\frac{1}{K}|\xi|^2 \leq \left\langle A(z) \xi, \xi \right\rangle \leq K |\xi|^2 \,\,\, \text{a.e. in}\,\,\, \Omega,
\end{equation}
for every $\xi \in \mathbb C$ and for some $1\leq K< \infty$.
The basic idea is that every positive quadratic form
\[
ds^2=a_{11}(x,y)dx^2+2a_{12}(x,y)dxdy+a_{22}(x,y)dy^2
\]
defined in a planar domain $\Omega$ can be reduced, by means of a quasiconformal change of variables, to the canonical form
\[
ds^2=\Lambda(du^2+dv^2),\,\, \Lambda\neq 0,\,\, \text{a.e. in}\,\, \widetilde{\Omega},
\]
given that $a_{11}a_{22}-a^2_{12} \geq \kappa_0>0$, $a_{11}>0$, almost everywhere in $\Omega$ \cite{Ahl66,BGMR}. Note that this fact can be extended to linear operators of the form $\textrm{div} [A(z) \nabla f(z)]$, $z=x+iy$, for matrix function $A \in M^{2 \times 2}(\Omega).$

Let $\xi(z)=\Real \varphi(z)$ be a real part of a quasiconformal mapping $\varphi(z)=\xi(z)+i \eta(z)$, which satisfies to the Beltrami equation:
\begin{equation}\label{BelEq}
\varphi_{\overline{z}}(z)=\mu(z) \varphi_{z}(z),\,\,\, \text{a.e. in}\,\,\, \Omega,
\end{equation}
where
$$
\varphi_{z}=\frac{1}{2}\left(\frac{\partial \varphi}{\partial x}-i\frac{\partial \varphi}{\partial y}\right) \quad \text{and} \quad
\varphi_{\overline{z}}=\frac{1}{2}\left(\frac{\partial \varphi}{\partial x}+i\frac{\partial \varphi}{\partial y}\right),
$$
with the complex dilatation $\mu(z)$ is given by
\begin{equation}\label{ComDil}
\mu(z)=\frac{a_{22}(z)-a_{11}(z)-2ia_{12}(z)}{\det(I+A(z))},\quad I= \begin{pmatrix} 1 & 0 \\ 0 & 1 \end{pmatrix}.
\end{equation}
We call this quasiconformal mapping (with the complex dilatation $\mu$ defined by (\ref{ComDil})) as an $A$-quasiconformal mapping.

Note that the uniform ellipticity condition \eqref{UEC} can be written as
\begin{equation}\label{OVCE}
|\mu(z)|\leq \frac{K-1}{K+1},\,\,\, \text{a.e. in}\,\,\, \Omega.
\end{equation}

Conversely we can obtain from \eqref{ComDil} (see, for example, \cite{AIM}, p. 412) that :
\begin{equation}\label{Matrix-F}
A(z)= \begin{pmatrix} \frac{|1-\mu|^2}{1-|\mu|^2} & \frac{-2 \Imag \mu}{1-|\mu|^2} \\ \frac{-2 \Imag \mu}{1-|\mu|^2} &  \frac{|1+\mu|^2}{1-|\mu|^2} \end{pmatrix},\,\,\, \text{a.e. in}\,\,\, \Omega.
\end{equation}

So, given any $A \in M^{2 \times 2}(\Omega)$, one produced, by \eqref{OVCE}, the complex dilatation $\mu(z)$, for which, in turn, the Beltrami equation \eqref{BelEq} induces a quasiconformal homeomorphism $\varphi:\Omega \to \widetilde{\Omega}$ as its solution, by the Riemann measurable mapping theorem (see, for example, \cite{Ahl66}). We will say that the matrix function $A$ induces the corresponding $A$-quasiconformal homeomorphism $\varphi$ or that $A$ and $\varphi$ are agreed.

So, by the given $A$-divergent form elliptic operator defined in a domain $\Omega\subset\mathbb C$ we construct so-called a $A$-quasiconformal mapping $\varphi:\Omega \to \widetilde{\Omega}$ with a quasiconformal coefficient
$$
K=\frac{1+\|\mu\mid L^{\infty}(\Omega)\|}{1-\|\mu\mid L^{\infty}(\Omega)\|},
$$
where $\mu$ defined by (\ref{ComDil}).

Note that the inverse mapping to the $A$-quasiconformal mapping $\varphi: \Omega \to \widetilde{\Omega}$ is the $A^{-1}$-quasiconformal mapping \cite{GPU2020}.

In \cite{GPU2020} was studied a connection between composition operators on Sobolev spaces and $A$-quasiconformal mappings.

\begin{theorem} \label{L4.1}
Let $\Omega,\widetilde{\Omega}$ be domains in $\mathbb C$. Then a homeomorphism $\varphi :\Omega \to \widetilde{\Omega}$ is an $A$-quasiconformal mapping
if and only if $\varphi$ induces, by the composition rule $\varphi^{*}(f)=f \circ \varphi$,
an isometry of Sobolev spaces $L^{1,2}_A(\Omega)$ and $L^{1,2}(\widetilde{\Omega})$:
\[
\|\varphi^{*}(f)\,|\,L^{1,2}_A(\Omega)\|=\|f\,|\,L^{1,2}(\widetilde{\Omega})\|
\]
for any $f \in L^{1,2}(\widetilde{\Omega})$.
\end{theorem}

This theorem generalizes the well known property of conformal mappings generate the isometry of uniform Sobolev spaces $L^1_2(\Omega)$ and $L^1_2(\widetilde{\Omega})$ (see, for example, \cite{C50}) and refines (in the case $n=2$) the functional characterization of quasiconformal mappings in the terms of isomorphisms of uniform Sobolev spaces \cite{VG75}.

\subsection{Weighted Sobolev-Poincar\'e inequality}

First of all, we recall that in \cite{GPU2019} was proved
the following non-weighted Sobolev-Poincar\'e inequality for a bounded domain $\widetilde{\Omega}\subset\mathbb C$.
\begin{theorem}
\label{PoinConst}
Let $\widetilde{\Omega}\subset\mathbb C$ be a bounded domain and $f \in W^{1,2}_0(\widetilde{\Omega})$. Then
\begin{equation}\label{InPS}
\|f \mid L^{r}(\widetilde{\Omega})\| \leq B_{r,2}(\widetilde{\Omega}) \|\nabla f \mid L^{2}(\widetilde{\Omega})\|, \,\,r \geq 2,
\end{equation}
where
\[
B_{r,2}(\widetilde{\Omega}) \leq \inf\limits_{p\in \left(\frac{2r}{r+2},2\right)}
\left(\frac{p-1}{2-p}\right)^{\frac{p-1}{p}}
\frac{\left(\sqrt{\pi}\cdot\sqrt[p]{2}\right)^{-1}|\widetilde{\Omega}|^{\frac{1}{r}}}{\sqrt{\Gamma(2/p) \Gamma(3-2/p)}}.
\]
\end{theorem}

According to this Sobolev-Poincar\'e inequality
and Theorem~\ref{L4.1} we obtain an universal weighted Sobolev-Poincar\'e inequality which holds in any simply connected planar domain with non-empty boundary. Denote by $h(z) =|J(z,\varphi)|$  the quasihyperbolic weight defined by an $A$-quasiconformal mapping $\varphi : \Omega \to \widetilde{\Omega}$.

\begin{theorem}\label{Th4.1}
Let $A$ belongs to a class  $M^{2 \times 2}(\Omega)$  and $\Omega$ be a simply connected planar domain.
Then for any function $f \in W^{1,2}_{0}(\Omega,A)$ the following weighted Sobolev-Poincar\'e inequality
\[
\left(\iint\limits_\Omega |f(z)|^rh(z)dxdy\right)^{\frac{1}{r}} \leq B_{r,2}(h,A,\Omega)
\left(\iint\limits_\Omega \left\langle A(z) \nabla f(z), \nabla f(z) \right\rangle dxdy\right)^{\frac{1}{2}}
\]
holds for any $r \geq 2$ with the constant $B_{r,2}(h,A,\Omega) = B_{r,2}(\widetilde{\Omega})$.
\end{theorem}

\begin{proof}
Given that $\Omega$ is a simply connected planar domain, then there exists \cite{Ahl66} an $\mu$-quasiconformal homeomorphism $\varphi : \Omega \to \widetilde{\Omega}$ with
\begin{equation*}
\mu(z)=\frac{a_{22}(z)-a_{11}(z)-2ia_{12}(z)}{\det(I+A(z))},
\end{equation*}
which is an $A$-quasiconformal mapping.

Hence by Theorem~\ref{L4.1} the equality
\begin{equation}\label{IN2.1}
||f \circ \varphi^{-1} \,|\, L^{1,2}(\widetilde{\Omega})|| = ||f \,|\, L^{1,2}(\Omega,A)||
\end{equation}
holds for any function $f \in L^{1,2}(\Omega,A)$.

Put $h(z):=|J(z,\varphi)|$.
Using the change of variable formula for the quasiconformal mappings \cite{VGR}, the equality \eqref{IN2.1} and Theorem~\ref{PoinConst}, we get that for any smooth function $f\in L^{1,2}(\Omega,A)$
\begin{multline*}
\left(\iint\limits_\Omega |f(z)|^rh(z)dxdy\right)^{\frac{1}{r}}
=\left(\iint\limits_\Omega |f(z)|^r |J(z,\varphi)| dxdy\right)^{\frac{1}{r}} \\
=\left(\iint\limits_{\widetilde{\Omega}} |f \circ \varphi^{-1}(w)|^rdudv\right)^{\frac{1}{r}}
\leq B_{r,2}(\widetilde{\Omega})
\left(\iint\limits_{\widetilde{\Omega}} \nabla (f \circ \varphi^{-1}(w))dudv\right)^{\frac{1}{2}} \\
= B_{r,2}(\widetilde{\Omega})
\left(\iint\limits_{\Omega} \left\langle A(z) \nabla g(z), \nabla f(z) \right\rangle dxdy\right)^{\frac{1}{2}}.
\end{multline*}

Approximating an arbitrary function $f \in W^{1,2}_{0}(\Omega,A)$ by smooth functions we have
$$
\left(\iint\limits_\Omega |f(z)|^rh(z)dxdy\right)^{\frac{1}{r}} \leq
B_{r,2}(h,A,\Omega) \left(\iint\limits_{\Omega} \left\langle A(z) \nabla f(z), \nabla f(z) \right\rangle dxdy\right)^{\frac{1}{2}},
$$
with the constant
$$
B_{r,2}(h,A,\Omega)=B_{r,2}(\widetilde{\Omega})  \leq \inf\limits_{p\in \left(\frac{2r}{r+2},2\right)}
\left(\frac{p-1}{2-p}\right)^{\frac{p-1}{p}}
\frac{\left(\sqrt{\pi}\cdot\sqrt[p]{2}\right)^{-1}|\widetilde{\Omega}|^{\frac{1}{r}}}{\sqrt{\Gamma(2/p) \Gamma(3-2/p)}}.
$$
\end{proof}

\subsection{ Estimates of Sobolev-Poincar\'e constants} In this section we consider (sharp) upper estimates of Sobolev-Poincar\'e constants in domains that satisfy the quasihyperbolic boundary condition.
The following theorem gives (sharp) upper estimates of (non-weighted) Sobolev-Poincar\'e constants in quasiconformal regular domains.
\begin{theorem}\label{Th4.3}
Let $A$ belongs to a class  $M^{2 \times 2}(\Omega)$ and a domain $\Omega$ be $A$-quasi\-con\-formal $\beta$-regular about $\widetilde{\Omega}$. Then:
\begin{enumerate}
\item for any function $f \in W^{1,2}_{0}(\Omega,A)$ and for any $s \geq 1$, the Sobolev-Poincar\'e inequality
\[
\|f\mid L^s(\Omega)\| \leq B_{s,2}(A,\Omega)
\|f\mid L^{1,2}_{A}(\Omega)\|
\]
holds with the constant
$$
B_{s,2}(A,\Omega) \leq B_{\frac{\beta s}{\beta-1},2}(\widetilde{\Omega}) \|J_{\varphi^{-1}}\mid L^{\beta}(\widetilde{\Omega})\|^{\frac{1}{s}}, \quad 1<\beta <\infty;
$$
\item for any function $f \in W^{1,2}_{0}(\Omega, A)$, the Sobolev-Poincar\'e inequality
\[
\|f\mid L^2(\Omega)\| \leq B_{2,2}(A,\Omega)
\|f\mid L^{1,2}_{A}(\Omega)\|
\]
holds with the constant
$$B_{2,2}(A,\Omega) \leq B_{2,2}(\widetilde{\Omega}) \big\|J_{\varphi^{-1}}\mid L^{\infty}(\widetilde{\Omega})\big\|^{\frac{1}{2}}, \quad \beta = \infty.
$$
Here $J_{\varphi^{-1}}$ is a Jacobian of the $A^{-1}$-quasiconformal mapping $\varphi^{-1}:\widetilde{\Omega}\to\Omega$.
\end{enumerate}
\end{theorem}

\begin{remark}
The constant $B_{2,2}^2(\widetilde{\Omega})=1/\lambda_1(\widetilde{\Omega})$, where $\lambda_1(\widetilde{\Omega})$ is the first Dirichlet eigenvalue of Laplacian in a domain $\widetilde{\Omega}\subset\mathbb C$.
\end{remark}

\section{Upper estimates for $\lambda_1(A,\Omega)$}

In this section we get upper estimates for the first Dirichlet eigenvalues of elliptic operators in divergence form in planar domains that satisfy the quasihyperbolic boundary condition \cite{KOT01,KOT02}.
 For this goal, we will use the spectral stability estimates have obtained in \cite{GPU2019}. Namely
\begin{theorem}\label{Stability}
Let $A$ belongs to a class  $M^{2 \times 2}(\Omega)$ and a domain $\Omega$ be $A$-quasiconformal $\beta$-regular about $\widetilde{\Omega}$.
Then for any $n\in \mathbb N$
\begin{multline*}
|\lambda_n[A, \Omega]-\lambda_n[A, \widetilde{\Omega}]|
\leq c_n A^2_{\frac{4\beta}{\beta -1},2}(\widetilde{\Omega}) \\
\times \left(|\widetilde{\Omega}|^{\frac{1}{2\beta}} +
\|J_{\varphi^{-1}}\,|\,L^{\beta}(\widetilde{\Omega})\|^{\frac{1}{2}} \right) \cdot
\|1-J_{\varphi^{-1}}^{\frac{1}{2}}\,|\,L^{2}(\widetilde{\Omega})\|,
\end{multline*}
where
$c_n=\max\left\{\lambda_n^2[A, \Omega], \lambda_n^2[A, \widetilde{\Omega}]\right\}$,
$J_{\varphi^{-1}}$ is a Jacobian of inverse mapping to $A$-quasiconformal homeomorphism $\varphi:\Omega \to \widetilde{\Omega}$, and
\[
A_{\frac{4\beta}{\beta -1},2}(\widetilde{\Omega}) \leq \inf\limits_{p\in \left(\frac{4\beta}{3\beta -1},2\right)}
\left(\frac{p-1}{2-p}\right)^{\frac{p-1}{p}}
\frac{\left(\sqrt{\pi}\cdot\sqrt[p]{2}\right)^{-1}|\widetilde{\Omega}|^{\frac{\beta-1}{4\beta}}}{\sqrt{\Gamma(2/p) \Gamma(3-2/p)}}~~.
\]
\end{theorem}

According to Theorem~\ref{Stability} and some classical results of the spectral theory of elliptic
operators, we get the following result.

\begin{theorem}\label{Th-5.2}
Let $A$ belongs to a class  $M^{2 \times 2}(\Omega)$ and $\Omega$ be an $A$-quasiconformal $\beta$-regular domain of area $\pi$.
Then we have
\begin{multline*}
\lambda_1(A,\Omega) \leq K \lambda_1(\mathbb D) + A^2_{\frac{4\beta}{\beta -1},2}(\mathbb D) K^2 \lambda_1^2(\mathbb D_\rho) \\
\times \left(\pi^{\frac{1}{2\beta}} +
\|J_{\varphi^{-1}}\,|\,L^{\beta}(\mathbb D)\|^{\frac{1}{2}} \right) \cdot
\|1-J_{\varphi^{-1}}^{\frac{1}{2}}\,|\,L^{2}(\mathbb D)\|,
\end{multline*}
where $\mathbb D_\rho$ is the largest disk inscribed in $\Omega$.
\end{theorem}

\begin{proof}
Because $\Omega$ is an $A$-quasiconformal $\beta$-regular domain, by Theorem~\ref{Stability} in the case $n=1$ and
$\widetilde{\Omega}=\mathbb D$, we have the following estimate:
\begin{multline*}
\lambda_1(A,\Omega) \leq \lambda_1(A, \mathbb D) + \max\left\{\lambda_1^2(A, \Omega), \lambda_1^2(A, \mathbb D)\right\} A^2_{\frac{4\beta}{\beta -1},2}(\mathbb D) \\
\times \left(\pi^{\frac{1}{2\beta}} +
\|J_{\varphi^{-1}}\,|\,L^{\beta}(\mathbb D)\|^{\frac{1}{2}} \right) \cdot
\|1-J_{\varphi^{-1}}^{\frac{1}{2}}\,|\,L^{2}(\mathbb D)\|.
\end{multline*}
Using the uniform ellipticity condition \eqref{UEC} we obtain
\begin{multline}\label{Ineq0}
\lambda_1(A,\Omega) \leq K \lambda_1(\mathbb D) + K^2 \max\left\{\lambda_1^2(\Omega), \lambda_1^2(\mathbb D)\right\} A^2_{\frac{4\beta}{\beta -1},2}(\mathbb D) \\
\times \left(\pi^{\frac{1}{2\beta}} +
\|J_{\varphi^{-1}}\,|\,L^{\beta}(\mathbb D)\|^{\frac{1}{2}} \right) \cdot
\|1-J_{\varphi^{-1}}^{\frac{1}{2}}\,|\,L^{2}(\mathbb D)\|.
\end{multline}
Further we indicate the maximum between $\lambda_1(\Omega)$ and $\lambda_1(\mathbb D)$.

Taking into account the Rayleigh-Faber-Krahn inequality (see, for example, \cite{GN13}), which states that the
disk minimizes the first Dirichlet eigenvalue among all planar domains of the same area,
i.e.,
\[
\lambda_1(\Omega) \geq \lambda_1(\mathbb D)=j_{0,1}^2,
\]
we get
\[
\max\left\{\lambda_1^2(\Omega), \lambda_1^2(\mathbb D)\right\}=\lambda_1^2(\Omega).
\]
Here $j_{0,1} \approx 2.4048$ is the first positive zero of the Bessel function $J_0$.

In turn, the property of monotonicity for the Dirichlet eigenvalues (see, for example, \cite{GN13}) implies the following estimate:
\begin{equation}\label{Mon}
\lambda_1(\Omega) \leq \lambda_1(\mathbb D_\rho)=\frac{j_{0,1}^2}{\rho^2}.
\end{equation}

Finally, combining inequalities \eqref{Mon}, \eqref{Ineq0}, we obtain the required result.
\end{proof}

Let $\varphi:\Omega\to\mathbb D$ be $A$-quasiconformal mappings. We note that there exist so-called volume-preserving maps, i.e. $|J(z,\varphi)|=1$, $z \in \Omega$. In this case, we call a domain $\Omega$ a volume-preserving $A$-quasiconformal $\beta$-regular domain.  For such domains we have:
\begin{corollary}\label{Cor5.3}
Let $A$ belongs to a class  $M^{2 \times 2}(\Omega)$ and $\Omega$ be a volume-preserving $A$-quasiconformal $\beta$-regular domain.
Then the following estimate holds
\[
\lambda_1(A,\Omega) \leq K \lambda_1(\mathbb D).
\]
\end{corollary}

As applications of Corollary \ref{Cor5.3} we consider several examples.

$\mathbf{Example \, 1.}$ The homeomorphism
$$
\varphi(z)= z e^{2i \log |z|},\,\, \varphi(0)=0, \quad z=x+iy,
$$
is $A$-quasiconformal and maps the unit disc $\mathbb D$
onto itself, transform radial lines into spiral infinitely winding around the origin.
The mapping $\varphi$ satisfies the Beltrami equation with
\[
\mu(z)=\frac{\varphi_{\overline{z}}}{\varphi_{z}}=\frac{i e^{2i \log |z|} \cdot z/\overline{z}}{(1+i)e^{2i \log |z|}}=\frac{1+i}{2}\frac{z}{\overline{z}}
\]
and the Jacobian $J(z,\varphi)=|\varphi_{z}|^2-|\varphi_{\overline{z}}|^2=1$.
Because
\[
\Real \mu(z) = \frac{x^2-y^2-2xy}{2(x^2+y^2)}, \quad \Imag \mu(z)=\frac{x^2-y^2+2xy}{2(x^2+y^2)}, \quad z=x+iy,
\]
we see that $\mu$ induces, by formula \eqref{Matrix-F}, the matrix function $A(z)$ with the following entries: 
\[
a_{11}=\frac{|1- \mu|^2}{1-|\mu|^2}=3- 2 \frac{x^2-y^2-2xy}{x^2+y^2},
\]
\[
a_{12}=a_{21}=\frac{-2 \Imag \mu}{1-|\mu|^2}=-2 \frac{x^2-y^2+2xy}{x^2+y^2},
\]
\[
a_{22}=\frac{|1+ \mu|^2}{1-|\mu|^2}=3+2 \frac{x^2-y^2-2xy}{x^2+y^2},
\]
or in the polar coordinates $z=\rho e^{i \theta}$ it has the form
$$
A=\begin{pmatrix} 3-2\sqrt{2} \cos(2\theta + \pi/4) & -2\sqrt{2} \sin(2\theta + \pi/4) \\ -2\sqrt{2} \sin(2\theta + \pi/4) & 3+2\sqrt{2} \cos(2\theta + \pi/4) \end{pmatrix}.
$$
Since $|J(w,\varphi^{-1})|=|J(z,\varphi)|^{-1}$ then by Corollary \ref{Cor5.3} we obtain
\[
\lambda_1(A,\mathbb D) \leq \frac{2+\sqrt{2}}{2-\sqrt{2}} \cdot \lambda_1(\mathbb D)=\frac{2+\sqrt{2}}{2-\sqrt{2}} \cdot j_{0,1}^2.
\]

$\mathbf{Example \, 2.}$ The homeomorphism
\[
\varphi(z)= \sqrt{a^2+1}z-a \overline{z}, \quad z=x+iy, \quad a\geq 0,
\]
is a $A$-quasiconformal and maps the interior of ellipse
$$
\Omega_e= \left\{(x,y) \in \mathbb R^2: \frac{x^2}{(\sqrt{a^2+1}+a)^2}+\frac{y^2}{(\sqrt{a^2+1}-a)^2}=1\right\}
$$
onto the unit disc $\mathbb D.$ The mapping $\varphi$ satisfies the Beltrami equation with
\[
\mu(z)=\frac{\varphi_{\overline{z}}}{\varphi_{z}}=-\frac{a}{\sqrt{a^2+1}}
\]
and the Jacobian $J(z,\varphi)=|\varphi_{z}|^2-|\varphi_{\overline{z}}|^2=1$.
It is easy to verify that $\mu$ induces, by formula \eqref{Matrix-F}, the matrix function $A(z)$ form
$$
A=\begin{pmatrix} (\sqrt{a^2+1}+a)^2 & 0 \\ 0 &  (\sqrt{a^2+1}-a)^2 \end{pmatrix}.
$$
Given that $|J(w,\varphi^{-1})|=|J(z,\varphi)|^{-1}$. Then by Corollary \ref{Cor5.3} we have
\[
\lambda_1(A,\Omega_e) \leq \frac{\sqrt{a^2+1}+a}{\sqrt{a^2+1}-a} \cdot \lambda_1(\mathbb D)= \frac{\sqrt{a^2+1}+a}{\sqrt{a^2+1}-a} \cdot j_{0,1}^2.
\]

$\mathbf{Example \, 3.}$ The homeomorphism
\[
\varphi(z)= \frac{z^{\frac{3}{2}}}{\sqrt{2} \cdot \overline{z}^{\frac{1}{2}}}-1,\,\, \varphi(0)=-1, \quad z=x+iy,
\]
is $A$-quasiconformal and maps the interior of the ``rose petal"
$$
\Omega_p:=\left\{(\rho, \theta) \in \mathbb R^2:\rho=2\sqrt{2}\cos(2 \theta), \quad -\frac{\pi}{4} \leq \theta \leq \frac{\pi}{4}\right\}
$$
onto the unit disc $\mathbb D$.
The mapping $\varphi$ satisfies the Beltrami equation with
\[
\mu(z)=\frac{\varphi_{\overline{z}}}{\varphi_{z}}=-\frac{1}{3}\frac{z}{\overline{z}}
\]
and the Jacobian $J(z,\varphi)=|\varphi_{z}|^2-|\varphi_{\overline{z}}|^2=1$.
We see that $\mu$ induces, by formula \eqref{Matrix-F}, the matrix function $A(z)$, which in the polar coordinates $z=\rho e^{i \theta}$ has the form
$$
A=\begin{pmatrix} 2\cos^2{\theta}+1/2\sin^2{\theta} & 3/4\sin{2\theta} \\ 3/4\sin{2\theta} & 1/2\cos^2{\theta}+2\sin^2{\theta} \end{pmatrix}.
$$
Since $|J(w,\varphi^{-1})|=|J(z,\varphi)|^{-1}$ then by Corollary \ref{Cor5.3} we obtain
\[
\lambda_1(A,\Omega_p) \leq 2\lambda_1(\mathbb D)=2 j_{0,1}^2.
\]

\section{Lower estimates for $\lambda_1(A,\Omega)$}

In this section we get lower estimates for the first Dirichlet eigenvalues of elliptic operators in divergence form in volume-preserving $A$-quasiconformal $\infty$-regular domains. For such domains Theorem~\ref{Th4.3} rewrite as

\begin{theorem}\label{Th5.1}
Let $A$ belongs to a class  $M^{2 \times 2}(\Omega)$ and $\Omega$ be a volume-preserving $A$-quasiconformal $\infty$-regular domain. Then
for any function $f \in W^{1,2}_{0}(\Omega, A)$, the Sobolev-Poincar\'e inequality
\[
\|f\mid L^2(\Omega)\| \leq B_{2,2}(A,\Omega)
\|f\mid L^{1,2}_{A}(\Omega)\|
\]
holds with the constant $B_{2,2}(A,\Omega) \leq B_{2,2}(\mathbb D)$.
\end{theorem}

 Here the constant $B_{2,2}^2(\mathbb D)=1/\lambda_1(\mathbb D)$, where $\lambda_1(\mathbb D)=j^2_{0,1}$ is the first Dirichlet eigenvalue of Laplacian in a disc $\mathbb D$.

Given this theorem we obtain the following assertion:

\begin{corollary}\label{Th5.2}
Let $A$ belongs to a class  $M^{2 \times 2}(\Omega)$ and $\Omega$ be a volume-preserving $A$-quasiconformal $\infty$-regular domain.
Then the following estimate holds
\[
\lambda_1(A,\Omega) \geq \lambda_1(\mathbb D).
\]

\end{corollary}

\begin{proof}
By the Min-Max principle and Theorem~\ref{Th5.1} we have
\[
\left(\iint\limits_{\Omega} |f(z)|^2 dxdy\right) \leq B^2_{2,2}(A,\Omega)
\iint\limits_\Omega \left\langle A(z) \nabla f(z), \nabla f(z) \right\rangle dxdy,
\]
where
$$
B_{2,2}(A,\Omega) \leq B_{2,2}(\mathbb D).
$$

Thus,
\[
\lambda_1(A,\Omega) \geq \lambda_1(\mathbb D).
\]

\end{proof}

\section{Spectral estimates in quasidiscs}

In this section we refine Theorem~\ref{Th-5.2} for Ahlfors-type domains (i.e. quasidiscs) $\Omega\subset\mathbb C$.
Recall that a domain $\Omega$ is called a $K$-quasidisc if it is the image of the unit disc $\mathbb D$ under a $K$-quasicon\-for\-mal homeomorphism of the plane onto itself. A domain $\Omega$ is a quasidisc if it is a $K$-quasidisc for some $K \geq 1$.

According to \cite{GH01}, the boundary of any $K$-quasidisc $\Omega$
admits a $K^{2}$-quasi\-con\-for\-mal reflection and thus, for example,
any quasiconformal homeomorphism $\psi:\mathbb{D}\to\Omega$ can be
extended to a $K^{2}$-quasiconformal homeomorphism of the whole plane
to itself.

Recall that for any planar $K$-quasiconformal homeomorphism $\psi:\Omega\rightarrow \widetilde{\Omega}$
the following sharp result is known: $J(z,\psi)\in L^p_{\loc}(\Omega)$
for any $1 \leq p<\frac{K}{K-1}$ (\cite{Ast,G81}).

Given the weak inverse H\"older inequality and the sharp estimates of the constants in doubling conditions for measures generated by Jacobians of quasiconformal mappings \cite{GPU17_2}, we obtain upper estimates of the first eigenvalue of linear
elliptic operators in divergence form with Dirichlet boundary conditions in Ahlfors type domains reformulated in terms of quasiconformal geometry of domains.

\begin{theorem}\label{Quasidisk}
Let $\Omega$ be a $K$-quasidisc of area $\pi$ and $\varphi:\Omega \to \mathbb D$ be an $A$-quasiconformal mapping. Assume that  $1<\beta<\frac{K}{K-1}$. Then
\begin{equation*}
\lambda_1(A,\Omega) \leq K \lambda_1(\mathbb D) + M_{\beta}(K)K^2 \lambda_1^2(\mathbb D_{\rho}) \|1-J_{\varphi^{-1}}^{\frac{1}{2}}\,|\,L^{2}(\mathbb D)\|,
\end{equation*}
where $\mathbb D_{\rho}$ is the largest disk inscribed in $\Omega$.
\end{theorem}

\begin{remark}
The quantity $M_{\beta}(K)$ in Theorem~\ref{Quasidisk} depends only on a quasiconformality
coefficient K of ${\Omega}$:
\begin{multline*}
M_{\beta}(K):=\inf\limits_{1< \beta <\beta^{*}} \Biggl\{
\inf\limits_{p\in \left(\frac{4 \beta}{3\beta -1},2\right)}
\left(\frac{p-1}{2-p}\right)^{\frac{2(p-1)}{p}}
\frac{\pi^{-\frac{\beta +1}{2\beta}} 4^{-\frac{1}{p}}}{\Gamma(2/p) \Gamma(3-2/p)} \\
\left(\frac{C_{\beta}K \pi^{\frac{1-\beta}{2\beta}}}{2}
\exp\left\{{\frac{K^2 \pi^2(2+ \pi^2)^2}{4\log3}}\right\}\cdot |{\Omega}|^{\frac{1}{2}}+\pi^{\frac{1}{2\beta}}\right) \Biggr\}, \\
C_\beta=\frac{10^{6}}{[(2\beta -1)(1- \nu(\beta))]^{1/2\beta}},
\end{multline*}
where $\beta^{*}=\min{\left(\frac{K}{K-1}, \widetilde{\beta}\right)}$, and $\widetilde{\beta}$ is the unique solution of the equation
$$\nu(\beta):=10^{8 \beta}\frac{2\beta -2}{2\beta -1}(24\pi^2K^2)^{2\beta}=1.
$$
The function $\nu(\beta)$ is a monotone increasing function. Thus for
any $\beta < \beta^{*}$ the number $(1- \nu(\beta))>0$ and $C_\beta > 0$.
\end{remark}

\begin{proof}
Since, for $K\geq 1$, $K$-quasidiscs are $A$-quasiconformal $\beta$-regular domains if $1<\beta<\frac{K}{K-1}$. Therefore, by Theorem~\ref{Th-5.2} for $1<\beta<\frac{K}{K-1}$ we have
\begin{multline}\label{Inequal_1}
\lambda_1(A,\Omega) \leq K \lambda_1(\mathbb D) + A^2_{\frac{4\beta}{\beta -1},2}(\mathbb D) K^2 \lambda_1^2(\mathbb D_\rho) \\
\times \left(\pi^{\frac{1}{2\beta}} +
\|J_{\varphi^{-1}}\,|\,L^{\beta}(\mathbb D)\|^{\frac{1}{2}} \right) \cdot
\|1-J_{\varphi^{-1}}^{\frac{1}{2}}\,|\,L^{2}(\mathbb D)\|.
\end{multline}
Now we estimate the quantity $\|J_{\varphi^{-1}}\,|\,L^{\beta}(\mathbb D)\|$. Taking into account (Corollary~5.2, \cite{GPU2020}) we obtain
\begin{multline}\label{Inequal_2}
\|J_{\varphi^{-1}}\,|\,L^{\beta}(\mathbb D)\| =
\left(\iint\limits_{\mathbb D} |J(w,\varphi^{-1})|^{\beta}~dudv \right)^{\frac{1}{\beta}} \\
\leq \frac{C^2_{\beta} K^2 \pi^{\frac{1-\beta}{\beta}}}{4} \exp\left\{{\frac{K^2 \pi^2(2+ \pi^2)^2}{2\log3}}\right\} \cdot |\Omega|.
\end{multline}
Finally, combining inequality \eqref{Inequal_1} with inequality \eqref{Inequal_2} and given that
\[
A_{\frac{4\beta}{\beta -1},2}^2(\mathbb D) \leq \inf\limits_{p\in \left(\frac{4 \beta}{3\beta -1},2\right)}
\left(\frac{p-1}{2-p}\right)^{\frac{2(p-1)}{p}}
\frac{\pi^{-\frac{\beta +1}{2\beta}} 4^{-\frac{1}{p}}}{\Gamma(2/p) \Gamma(3-2/p)}
\]
we get the desired result.
\end{proof}

\textbf{Acknowledgements.} The first author was supported by the United States-Israel Binational Science Foundation (BSF Grant No. 2014055). The second author was supported by RSF Grant No. 20-71-00037 (Sections 3, 4).

\vskip 0.3cm

Department of Mathematics, Ben-Gurion University of the Negev, P.O.Box 653, Beer Sheva, 8410501, Israel

\emph{E-mail address:} \email{vladimir@math.bgu.ac.il} \\

 Division for Mathematics and Computer Sciences, Tomsk Polytechnic University, 634050 Tomsk, Lenin Ave. 30, Russia; Department of Mathematical Analysis and Theory of Functions, Tomsk State University, 634050 Tomsk, Lenin Ave. 36, Russia
							
 \emph{E-mail address:} \email{vpchelintsev@vtomske.ru}   \\
			
	Department of Mathematics, Ben-Gurion University of the Negev, P.O.Box 653, Beer Sheva, 8410501, Israel
							
	\emph{E-mail address:} \email{ukhlov@math.bgu.ac.il}

\end{document}